\documentclass[11pt,reqno]{amsart}

\usepackage{amsmath, amsfonts, amsthm, amssymb}

\newtheorem{Theorem}{Theorem}[section]
\newtheorem{Lemma}{Lemma}[section]

\theoremstyle{definition}
\newtheorem{Definition}{Definition}[section]

\theoremstyle{remark}
\newtheorem{Remark}{Remark}[section]

\numberwithin{equation}{section}

\def\va{\varphi}
\renewcommand{\u}{{\bf u}}

\newcommand{\R}{{\mathbb R}}
\newcommand{\Dv}{{\nabla\,\cdot\,}}

\newcommand{\x}{{\bf x}}

\newcommand{\dl}{\delta}

\def\f{\frac}
\renewcommand{\O}{\Omega}

\def\D{\Delta }

\def\hf1{^\f{1}{1-\xi^2}}

\def\be{\begin{equation}}
\def\en{\end{equation}}
\def\bs{\begin{split}}
\def\es{\end{split}}

\renewcommand{\d}{{\bf d}}
\newcommand{\F}{{\bf d}}

\author{Xiaoli Li and\ Dehua Wang}
\address{Department of Mathematical Sciences,  Tsinghua University,
                           Beijing 100084, China; and
                           Department of Mathematics, University of Pittsburgh,
                           Pittsburgh, PA 15260, USA.}
\email{xllithu@gmail.com}

\address{Department of Mathematics, University of Pittsburgh,
                           Pittsburgh, PA 15260, USA.}
\email{dwang@math.pitt.edu}

\title[Global Solution to Liquid Crystal Flow]
{Global Solution to the Incompressible \\ Flow of Liquid
 Crystals}


\date{\today}

\begin{document}

\begin{abstract}
The initial boundary value problem for the three-dimensional incompressible flow of liquid crystals is considered in a bounded smooth domain.
The existence and uniqueness is established for both the local
strong solution with large initial data and the global strong
solution with small data.  It is also proved that when the strong
solution exists, a weak solution must be equal to the unique strong
solution with the same data.
\end{abstract}
\maketitle

\section{Introduction}

Liquid crystals  are a state of matter that have properties
between those of a conventional liquid and those of a solid crystal
that are optically anisotropic, even when they are at rest. In this
work, we are interested in a Navier-Stokes type model for
incompressible fluids that takes into account the crystallinity of
the fluid molecules in the three-dimensional case, that is, a nematic liquid
crystal model, which can be governed by the following nonlinear
hydrodynamical system (see \cite{JLE,Leslie1,Lin2} and references
therein):
\begin{subequations}\label{e2}
\begin{align}
&\f{\partial\u}{\partial t}+\u\cdot\nabla\u-\mu\D\u+\nabla P
       =-\lambda\nabla\cdot\left(\nabla \d\odot\nabla\d\right), \label{e21}\\
&\f{\partial\d}{\partial t}+\u\cdot\nabla\d=\gamma\left(\D\d+|\nabla\d|^2\d\right), \label{e22}\\
&\Dv\u=0, \label{e23}
\end{align}
\end{subequations}
where $\u\in\R^3$ denotes the velocity, $\d\in \mathbb{S}^2$ (the unit sphere in $\R^3$)  the unit-vector field that represents the
 macroscopic/continuum molecular orientations,
 $P\in\R$ is the pressure (including both the hydrostatic part and the induced elastic part from the orientation field) arising from
  the incompressibility $\nabla\cdot\u=0$;  and they all depend on the spatial variable
 $\x=(x_1,x_2,x_3)\in\R^3$ and the time variable $t>0$.
 The term $\lambda\nabla\cdot\left(\nabla\d\odot\nabla\d\right)$ in the stress tensor represents the
 anisotropic feature of the system.
 The positive constants $\mu, \lambda, \gamma$ stand for viscosity, the competition between kinetic energy and potential energy,
 and microscopic elastic relaxation time or the Debroah number for the molecular orientation field, respectively. We set these three constants to be one since their exact values do not play any role in our analysis.
The symbol $\nabla\d\odot\nabla\d$ denotes a matrix whose $(i,j)$-th
entry is $\partial_{x_i} \d\cdot\partial_{x_j}\d$ for $1\leq i,j\leq
3$, and it is easy to see that $\nabla\d\odot\nabla\d=(\nabla
\d)^\top\nabla\d,$ where $(\nabla \d)^\top$ denotes the transpose of
the $3\times 3$ matrix $\nabla \d$.

System \eqref{e2} is a simplified version, but still retains most of
the interesting mathematical properties (without destroying the
basic nonlinear structure) of the original Ericksen-Leslie model
(\cite{Eri, Eri2, HK, HKL, Leslie1, Lin2}) for the hydrodynamics of
nematic liquid crystals; see \cite{LL,LW,SL} for more discussions on
the relations of the two models. Both the Ericksen-Leslie system and
the simplified one describe the macroscopic continuum time evolution
of liquid crystal materials under the influence of both the velocity
$\u$ and the orientation  $\d$ which can be derived
from the averaging/coarse graining of the directions of rod-like
liquid crystal molecules. In particular, there is a force term in
the $\u$-system depending on $\d$; the left hand side of the
$\d$-system stands for the kinematic transport by the flow field
while the right hand side represents the internal relaxation due to
the elastic energy.
In many situations, the flow velocity field does disturb the alignment
of the molecule, and in turn, a change in the alignment will induce
velocity.  

We consider the initial-boundary value problem of system \eqref{e2}
in a bounded smooth domain $\O\subset\R^3$ with the initial condition:
\begin{equation}\label{ic}
(\u, \d)\mid_{t=0}=(\u_0(\x), \d_0(\x)), \quad  \x\in {\Omega},
\end{equation}
and the boundary condition:
\begin{equation}\label{bc}
(\u,\partial_{\bf \nu}\d)\mid_{\partial\Omega}=(0,0), \quad t>0,
\end{equation}
where ${\bf \nu}$ is the outer unit-normal vector field on $\partial\O$,
$\u_0: \O\rightarrow\R^3$, and $\d_0: \O\rightarrow \mathbb{S}^2$  are given with compatability;
 for the velocity $\u$ the non-slip boundary condition, i.e., homogeneous
Dirichlet type, is considered,  and for the orientation vector $\d$ the homogeneous Neumann
boundary condition 
is posed here.

Roughly speaking, system \eqref{e2} is a coupling between the incompressible Navier-Stokes equations
and the transported flow of harmonic maps.
There have been many studies on system \eqref{e2}, see \cite{HK, HKL, Lin2, LL, LL2, LLW, LS, LW, SL, XZL} and the references therein. Recently, in Lin-Lin-Wang \cite{LLW}, they established both interior and boundary regularity theorem for such a system in dimension two under smallness conditions. And, they also established the existence of global weak solutions that are smooth away from at most finitely many singular times in any bounded smooth domain of $\R^2$.
In Lin-Liu \cite{LL}, they addressed both the regularity and existence of global weak solutions to the $n$-dimensional ($n=2,3$) Leslie system of variable length, i.e.,  the Dirichlet energy
$$\frac{1}{2}\int_\O|\nabla\d|^2 d\x\quad \text{for}\quad \d: \O\rightarrow \mathbb{S}^{n-1}$$
 is replaced by the Ginzburg-Landau energy
 $$\int_\O\left(\frac{1}{2}|\nabla\d|^2+\frac{(1-|\d|^2)^2}{4\varepsilon^2}\right) d\x \quad \text{for}\quad
 \d: \O\rightarrow \R^n.$$
  More precisely, they proved the global existence of weak solutions with large initial
data under the assumptions that $\u_0\in L^2(\O), \d_0\in H^1(\O)$ with $\d_0|_{\partial\Omega}\in
H^{\frac{3}{2}}(\partial\Omega)$ in dimension two and three.  The existence and uniqueness of
global classical solution was also obtained if $\u_0\in H^1(\O), \d_0\in
H^2(\O)$ in dimension two or dimension three when the fluid viscosity $\mu$ is large enough.
The similar results were obtained also in \cite{SL} for a different but similar model. When weak solutions are
discussed, the partial regularity theorem of the weak solution was investigated in
\cite{LL2} (and also \cite{HKL}), similar to the classical theorem by Caffarelli-Kohn-Nirenburgh \cite{CKN} for the Navier-Stokes equations.

In this paper, we are interested in the existence and uniqueness of
global strong solution $(\u,\d, P) $ of \eqref{e2} in
$W^{2,q}(\O)^3\times W^{3,q}(\O)^3\times W^{1,q}(\O)$ with $q>3$. By
a $\textit{Strong Solution}$, we mean a triplet $(\u, \d, P)$
satisfying \eqref{e2} almost everywhere with the initial-boundary
conditions \eqref{ic}-\eqref{bc}. Our strategy to consider \eqref{e2}
is to linearize it as
\begin{subequations}\label{e4}
\begin{align}
&\f{\partial\u}{\partial t}-\D\u+\nabla
P=-{\bf v}\cdot\nabla{\bf v}-\nabla\cdot\big((\nabla{\bf f})^\top\nabla{\bf f}\big), \label{e41}\\
&\f{\partial\d}{\partial
t}-\D\d=-{\bf v}\cdot\nabla{\bf f}+|\nabla{\bf f}|^2{\bf f},\label{e42}\\
&\Dv\u=0,  \label{e43}
\end{align}
\end{subequations}
for some given functions ${\bf v}\in\R^3$ and ${\bf f}\in\R^3$. One of
the motivations of making such an linearization is that we can use
the maximal regularity of Stokes equations (cf. Theroem \ref{T4}) and the
parabolic equation (cf. Theroem \ref{T3}). We first use an iteration method to
establish the local existence and uniqueness of strong solution with
general large initial data. Then we prove the global existence by
establishing some global estimates under the condition that the
initial data are small in some norm. As system \eqref{e2} contains the Navier-Stokes equations as a subsystem, one cannot expect generally  better results than those for the Navier-Stokes equations. The uniqueness of global weak solution is still an open
problem. We shall prove that when the strong solution exists, all
the global weak solutions must be equal to
the unique strong solution, which is called the weak-strong
uniqueness. Similar results were obtained by Danchin \cite{D} for
the density-dependent incompressible Navier-Stokes equations. We
shall establish our results in the spirit of \cite{D}, while
developing new estimates for the orientation field $\d$. Due to the specific structure of the equations for $\u$, especially the strongly nonlinear term $(\nabla\d)^\top\triangle\d$ in the $\u$-system, it will be necessary to obtain more regularity for $\d$.

The rest of the  paper is organized as follows. In Section 2, we state
our main results on local and global existence of strong solution, as well as the weak-strong uniqueness.
In Section 3, we recall  the maximal regularity for Stokes equations and the parabolic equation, and also some $L^\infty$ estimates.
In Section 4, we give the proof of the local existence. In Section 5, we prove the global existence.
Finally in Section 6, we  show the weak-strong uniqueness.
\bigskip

\section{Main Results}

In this section, we state our main results.
If $k>0$ is an integer and $p\ge 1$, we denote by $W^{k,p}$ the set of functions in $L^p(\O)$ whose derivatives of up to order
$k$ belong to $L^p(\O)$.
For $T>0$ and a function space $X$, denote by $L^p(0,T; X)$ the set of Bochner measurable X-valued time dependent
functions $f$ such that $t\rightarrow \|f\|_{X}$ belongs to $L^p(0,T)$.
Let us  define the functional spaces in which the existence of solutions
is going to be obtained:

\begin{Definition}\label{df1}
For $T>0$ and $1<p, q<\infty$, we denote by $M^{p,q}_T$ the set of
triplets $(\u,\d, P)$ such that
$$\u\in C([0,T]; D_{A_q}^{1-\f{1}{p},p})\cap L^p(0,T;
W^{2,q}(\O)\cap W^{1,q}_0(\O)), \ \ \frac{\partial\u}{\partial t}\in
L^p(0,T; L^q(\O)),\ \ \Dv\u=0,$$
$$\d\in C([0,T]; B_{q,p}^{3(1-\f{1}{p})})\cap L^p(0,T;W^{3,q}(\O)), \ \ \frac{\partial\d}{\partial t}\in L^p(0,T; L^q(\O)),$$
and
$$P\in L^p(0,T; W^{1,q}(\O)), \ \ \int_\O P \ d\x=0.$$
The corresponding norm is denoted by $\|\cdot\|_{M^{p,q}_T}$.
\end{Definition}

We remark that the condition $\int_\O P\ d\x=0$ in the definition
\eqref{df1} holds automatically if we replace $P$ by
$$P-\f{1}{|\O|}\int_\O P\ d\x$$ in \eqref{e2}.
Also, in the above definition, the space $D_{A_q}^{1-\f{1}{p},p}$
stands for some fractional domain of the Stokes operator in $L^q$
(cf. Section 2.3 in \cite{D}). Roughly, the vector-fields of
$D_{A_q}^{1-\f{1}{p},p}$ are vectors which have  $2-\f{2}{p}$
derivatives   in $L^q$, are divergence-free, and vanish on
$\partial\O$. The Besov space (for definition, see \cite{BL})
$B_{q,p}^{3(1-\f{1}{p})}$ can be regarded as the interpolation space
between $L^q$ and $W^{3,q}$, that is,
$$B_{q,p}^{3(1-\f{1}{p})}=(L^q, W^{3,q})_{1-\f{1}{p},p}.$$ Moreover,  we note that
$B_{q,p}^{3(1-\f{1}{p})}\hookrightarrow W^{1,q}$ if $p\geq \frac{3}{2}$. By the imbedding
$W^{1,q}\hookrightarrow L^\infty \ \textrm {as}\  q>3$, one has $B_{q,p}^{3(1-\f{1}{p})}\hookrightarrow L^\infty$, which will be used repeatly in this paper.

\vspace{2mm}

The local existence will be shown by using an iterative method, and
if the initial data are sufficiently small in some suitable function
spaces, the solution is indeed global in time. More precisely, our
existence result reads:

\begin{Theorem}\label{T1}
Let $\O$ be a bounded smooth domain in $\R^3$. Assume $\frac{3}{2}\leq p,q<\infty$ with $\f{p}{2}(1-\f{3}{q})\in
(0,1)$ and $\u_0\in D_{A_q}^{1-\f{1}{p},p}, \, \d_0\in
B_{q,p}^{3(1-\f{1}{p})}$. Then,
\begin{enumerate}
\item There exists $T_0>0$,  such that, system \eqref{e2} with the initial-boundary conditions \eqref{ic}-\eqref{bc} has a
unique strong solution $(\u, \d, P)\in M^{p,q}_{T_0}$ in $\O\times
(0, T_0)$;
\item Moreover,  there exists  $\delta_0>0$, such that, if the initial data  satisfy
$$\|\u_0\|_{D_{A_q}^{1-\f{1}{p},p}}\le\delta_0, \quad
\|\d_0\|_{B_{q,p}^{3(1-\f{1}{p})}}\le\delta_0,$$ then
\eqref{e2}-\eqref{bc} has a unique strong solution $(\u, \d, P)\in
M^{p,q}_{T}$ in $\O\times (0, T)$ for all  $T>0$.
\end{enumerate}
\end{Theorem}

According to \cite{LLW},  a $\textit{Weak Solution}$ to \eqref{e2}
with the initial-boundary conditions \eqref{ic}-\eqref{bc} means a
triplet $(\tilde{\u}, \tilde{\bf d}, \Pi )$ satisfying system
\eqref{e2} in $\O\times(0,T)$ for $0<T\leq +\infty$ in the sense of
distributions, i.e, for any smooth function $\psi(t)$ with
$\psi(T)=0$ and $\phi(\x)\in \big(H_0^1(\O)\big)^3$ with
$\Dv\phi=0$, we have
\begin{equation*}
\begin{split}
&-\int_0^T (\tilde{\u}, \psi'\phi) \,dt+\int_0^T
(\tilde{\u}\cdot\nabla\tilde{\u}, \psi\phi)
\,dt+\mu\int_0^T(\nabla\tilde{\u}, \psi\nabla\phi)
\,dt\\
&=\psi(0)(\u_0, \phi)+\lambda\int_0^T(\nabla\tilde{\bf d}\odot
\nabla\tilde{\bf d}, \psi\nabla\phi) \,dt,
\end{split}
\end{equation*}
and
\begin{equation*}
\begin{split}
&-\int_0^T(\tilde{\d},\psi'\phi) \,dt+\int_0^T(\tilde{\u}\cdot\nabla
\tilde{\bf
d}, \psi\phi)\ dt+\gamma\int_0^T(\nabla\tilde{\bf d}, \psi\nabla\phi)\ dt\\
&=\psi(0)(\d_0, \phi)+\gamma\int_0^T|\nabla\tilde{\d}|^2(\tilde{\d},
\psi\phi)\ dt,
\end{split}
\end{equation*}
where $(\cdot,\cdot)$ denotes the inner product in $L^2(\O)^3$.
Moreover, $(\tilde{\u}, \tilde{\bf d})$ satisfies \eqref{bc} in the
sense of trace. 



Next, we will give a uniqueness result.
For $0<T<+\infty$, suppose $(\tilde{\u}, \tilde{\d}, \Pi)$ with
$$\tilde{\u}\in L^{2,\infty}(\O\times [0,T])\cap W_2^{1,0}(\O_T),\quad
\tilde{\d}\in L^\infty ([0,T],H^1(\O))\cap L^2([0,T],H^2(\O)),$$
and $$\nabla\Pi\in L^{\frac{4}{3}}(0,T; L^{\frac{6}{5}}(\O))$$ is a global weak
solution to \eqref{e2}-\eqref{bc}.  Since $\nabla\tilde{\d}\in
L^2(0,T;H^1(\O))$ and $|\tilde{\d}|=1$, then
$$\triangle\tilde{\d}\cdot\tilde{\d}+|\nabla\tilde{\d}|^2=0.$$
 Hence
$|\nabla\tilde{\d}|\in L^4(\O\times[0,T])$. We have the following
energy inequality (cf. \cite{LLW},  Section 5 for the two-dimensional
case):
\begin{equation}\label{26}
\begin{split}
&\frac{1}{2}\int_\O(|\tilde{\u}(t)|^2+|\nabla\tilde{\d}(t)|^2)
d\x+\int_0^t\!\!\int_\O
(|\nabla\tilde{\u}|^2+|\triangle\tilde{\d}+|\nabla\tilde{\d}|^2\tilde{\d}|^2)
d\x ds\\& \leq\frac{1}{2}\int_\O(|\u_0|^2+|\nabla\d_0|^2)d\x,
\end{split}
\end{equation}
for all $t\in (0,\infty)$.
We remark that the assumption on pressure function  holds since 
$\Pi$ can be determined as in the Navier-Stokes equations (see
\cite{Galdi}).


As for the standard Navier-Stokes equations,
 the question of uniqueness in the above class  remains open.
 However,  for the same initial-boundary conditions, a relation between the
weak solution and the strong solution can be formulated as:

\begin{Theorem}\label{T2}
Assume that $\u_0\in D_{A_q}^{1-\f{1}{p},p}$ and $\d_0\in
B_{q,p}^{3(1-\f{1}{p})}$. Then any weak
solution to \eqref{e2}-\eqref{bc} in the above class is unique and
indeed is equal to its unique strong solution.
\end{Theorem}

Usually, we call this kind of uniqueness as $\textit{Weak-Strong
Uniqueness}$. For the similar results on the compressible
Navier-Stokes equations, we refer the  readers to \cite{EB, L}.

\bigskip

\section{Maximal Regularity}

In this section, we recall the maximal regularities for the parabolic operator and the Stokes operator, as well as some $L^\infty$ estimates.

For $T>0$, $1<p,q<\infty$, denote
$$\mathcal{W}(0,T):=\big(W^{1,p}(0,T; L^q(\O))\big)^3\cap \big(L^p(0,T;W^{3,q}(\O))\big)^3.$$
Throughout this paper, $C$ stands for a generic positive constant.

We first recall the maximal regularity for the parabolic operator (cf.
Theorem 4.10.7 and Remark 4.10.9 in \cite{HA}):
\begin{Theorem}\label{T3}
Given $1<p, q<\infty$, $\omega_0\in B_{q,p}^{3(1-\f{1}{p})}$ and
$f\in \big(L^p(0,T; L^q(\R^3))\big)^3$, the Cauchy problem
\begin{equation*}
\begin{cases}
\frac{\partial\omega}{\partial t}-\D\omega=f, \quad t>0, \\
\omega(0)=\omega_0
\end{cases}
\end{equation*}
has a unique solution $\omega\in \mathcal{W}(0,T)$, and
$$\|\omega\|_{\mathcal{W}(0,T)}\le C\left(\|f\|_{L^p(0,T; L^q(\R^3))}+\|\omega_0\|_{B_{q,p}^{3(1-\f{1}{p})}}\right),$$
where $C$ is independent of $\omega_0$, $f$ and $T$. Moreover,
there exists a positive constant $c_0$ independent of $f$ and $T$
such that
$$\|\omega\|_{\mathcal{W}(0,T)}\ge c_0\sup_{t\in(0,T)}\|\omega(t)\|_{B_{q,p}^{3(1-\f{1}{p})}}.$$
\end{Theorem}

Now we recall the maximal regularity  for the Stokes equations (cf. Theorem 3.2 in \cite{D}):

\begin{Theorem}\label{T4}
Let $\O$ be a bounded domain with $C^{2+\varepsilon}$ boundary in
$\R^3$ and $1<p,q<\infty$. Assume that $\u_0\in
D_{A_q}^{1-\f{1}{p},p}$ and $f\in L^p(\R^{+}; L^q(\O))$. Then the system
\begin{equation*}
\begin{cases}
\frac{\partial\u}{\partial t}-\D\u+\nabla P=f,\quad \int_\O P\ d\x=0,\\
\Dv\u=0,\quad \u|_{\partial\O}=0,\\
\u|_{t=0}=\u_0
\end{cases}
\end{equation*}
has a unique solution $(\u, P)$ satisfying the following
inequality for all $T>0$:
\begin{equation}\label{41}
\begin{split}
\|\u(T)\|_{D_{A_q}^{1-\f{1}{p},p}}+&\left(\int_0^T\left\|\left(\nabla
P, \D\u,
\frac{\partial \u}{\partial t}\right)\right\|^p_{L^q}dt\right)^{\f{1}{p}}\\
&\le
C\left(\|\u_0\|_{D_{A_q}^{1-\f{1}{p},p}}+\left(\int_0^T\|f(t)\|_{L^q}^pdt\right)^{\f{1}{p}}\right)
\end{split}
\end{equation}
with $C=C(q, p,\sigma(\O))$ where $\sigma(\O)$ stands for the
open set $$\sigma(\O)=\left\{\frac{\x}{\dl(\O)}|
\x\in \O\right\}$$ with $\dl(\O)$ denoting the diameter of $\O$.
\end{Theorem}

\begin{Remark}
We notice that  \eqref{41} does not include the estimate for
$\|\u\|_{L^p(0,T; L^q)}$. Indeed,  thanks to $\u|_{\partial\O}=0$,
Poincare's inequality, and the fact $\int_\O\nabla\u \ d\x=0$, we
have
$$\|\u\|_{W^{2,q}}\le C\|\D\u\|_{L^q},$$ and then \eqref{41} can
be rewritten as
\begin{equation*}
\begin{split}
\|\u(T)\|_{D_{A_q}^{1-\f{1}{p},p}}+&\left(\int_0^T\left\|\left(\nabla
P, \u, \D\u,
\partial_t\u\right)\right\|^p_{L^q}dt\right)^{\f{1}{p}}\\
&\le
C\left(\|\u_0\|_{D_{A_q}^{1-\f{1}{p},p}}+\left(\int_0^T\|f(t)\|_{L^q}^pdt\right)^{\f{1}{p}}\right).
\end{split}
\end{equation*}
\end{Remark}

We have the $L^\infty$ estimate in the spatial variable as follows (cf. Lemma 4.1 in \cite{D}).

\begin{Lemma}\label{l1}
Let $1<p,q,r,s<\infty$ satisfy
$0<\f{p}{2}-\f{3p}{2q}<1,$
 then
$$\|\nabla f\|_{L^p(0,T; L^\infty)}\le
CT^{\f{1}{2}-\f{3}{2q}}\|f\|^{1-\theta}_{L^\infty(0,T;
D_{A_q}^{1-\f{1}{p},p})}\|f\|^\theta_{L^p(0,T; W^{2,q})},$$
 for some
constant $C$ depending only on $\O, p, q$ and
$$\f{1-\theta}{p}=\f{1}{2}-\f{3}{2q}.$$
\end{Lemma}
Similarly, we have,
\begin{Lemma}\label{l2}
Let $1<p,q<\infty$ satisfy $0<\f{p}{2}-\f{3p}{2q}<1$, then
$$\|\nabla f\|_{L^p(0,T; L^\infty)}\le
CT^{\f{1}{2}-\f{3}{2q}}\|f\|^{1-\theta}_{L^\infty(0,T;
B_{q,p}^{2(1-\f{1}{p})})} \|f\|^\theta_{L^p (0,T; W^{2,q})},$$ for
some constant $C$ depending only on $\O, p, q$ and
$$\f{1-\theta}{p}=\f{1}{2}-\f{3}{2q}.$$
\end{Lemma}
\begin{proof}
First, we notice that (cf. Theorem
6.4.5 in \cite{BL})
$$(B^{1-\f{2}{p}-\f{3}{q}}_{\infty,\infty}, \
B^{1-\f{3}{q}}_{\infty,\infty})_{\theta,1}=B^0_{\infty,1}\ \
\textrm{with} \ \  \f{1-\theta}{p}=\f{1}{2}-\f{3}{2q}.$$  Also the imbedding
$B^0_{\infty,1}\hookrightarrow L^\infty$ holds due to Theorem
6.2.4 in \cite{BL}. Hence,
\begin{equation}\label{5555}
|\nabla f|_\infty\le C\|\nabla f\|_{B^0_{\infty,1}}\le C\|\nabla
f\|^{\theta}_{B^{1-\f{3}{q}}_{\infty,\infty}}\|\nabla
f\|^{1-\theta}_{B^{1-\f{2}{p}-\f{3}{q}}_{\infty,\infty}}.
\end{equation}
We remark that (cf.
Theorem 6.2.4 and Theorem 6.5.1 in \cite{BL}) $$ B_{q,p}^{2(1-\f{1}{p})}\hookrightarrow
B_{\infty,\infty}^{2-\f{2}{p}-\f{3}{q}}, \quad W^{1,q}\hookrightarrow
B_{q,\infty}^1\hookrightarrow B_{\infty,\infty}^{1-\f{3}{q}}.$$
Thus, according to \eqref{5555} and by applying H\"{o}lder's inequality, we
deduce that
\begin{equation*}
\begin{split}
\|\nabla f\|_{L^p(0,T;  L^\infty)}&\le C\left(\int_0^T\|\nabla
f\|^{p\theta}_{B^{1-\f{3}{q}}_{\infty,\infty}}\|\nabla
f\|^{p(1-\theta)}_{B^{1-\f{2}{p}-\f{3}{q}}_{\infty,\infty}}dt\right)^{\f{1}{p}}\\
&\le
C\left(\int_0^T\|f\|^{p\theta}_{W^{2,q}}\|f\|^{p(1-\theta)}_{B_{q,p}^{2-\frac{2}{p}-\frac{3}{q}}}dt\right)^{\f{1}{p}}\\
&\le
C\left(\int_0^T\|f\|^{p\theta}_{W^{2,q}}\|f\|^{p(1-\theta)}_{B_{q,p}^{2(1-\f{1}{p})}}dt\right)^{\f{1}{p}}\\
&\le CT^{\f{1}{2}-\f{3}{2q}}\|f\|^{1-\theta}_{L^\infty(0,T;
B_{q,p}^{2(1-\f{1}{p})})}\|f\|^\theta_{L^p(0,T; W^{2,q})}.
\end{split}
\end{equation*}
The proof is complete.
\end{proof}
\begin{Remark}
In this paper, we will use the following weaker result:
$$\|\nabla f\|_{L^p(0,T; L^\infty)}\le
CT^{\f{1}{2}-\f{3}{2q}}\|f\|^{1-\theta}_{L^\infty(0,T;
B_{q,p}^{3(1-\f{1}{p})})} \|f\|^\theta_{L^p (0,T; W^{3,q})}.$$
\end{Remark}
\begin{Lemma}\label{l3}
Let $1<p,q<\infty$ satisfy $0<\f{p}{3}-\f{p}{q}<1$, then
$$\|\triangle f\|_{L^p(0,T; L^\infty)}\le
CT^{\f{1}{3}-\f{1}{q}}\|f\|^{1-\theta}_{L^\infty(0,T;
B_{q,p}^{3(1-\f{1}{p})})} \|f\|^\theta_{L^p(0,T; W^{3,q})},$$ for
some constant $C$ depending only on $\O, p, q$ and
$$\f{1-\theta}{p}=\f{1}{3}-\f{1}{q}.$$
\end{Lemma}
\begin{proof}
First, we notice that (cf. Theorem 6.4.5 in
\cite{BL})
$$(B^{2-\f{3}{p}-\f{3}{q}}_{\infty,\infty}, B^{2-\f{3}{q}}_{\infty,\infty})_{\theta,1}=B^1_{\infty,1}\ \textrm{with}
\f{1-\theta}{p}=\f{1}{3}-\f{1}{q}.$$ Also the imbedding $B^1_{\infty,1}\hookrightarrow W^{1,
\infty}$ is true due to Theorem 6.2.4 in \cite{BL}. Hence,
\begin{equation}\label{5}
\|\triangle f\|_{L^\infty}\le C\|\nabla f\|_{W^{1, \infty}}\le
C\|\nabla f\|_{B^1_{\infty,1}}\le C\|\nabla
f\|^{\theta}_{B^{2-\f{3}{q}}_{\infty,\infty}}\|\nabla
f\|^{1-\theta}_{B^{2-\f{3}{p}-\f{3}{q}}_{\infty,\infty}}.
\end{equation}
We remark that (cf.
Theorem 6.2.4 and Theorem 6.5.1 in \cite{BL}) $$B_{q,p}^{3(1-\f{1}{p})}\hookrightarrow
B_{\infty,\infty}^{3-\f{3}{p}-\f{3}{q}}, \quad W^{2,q}\hookrightarrow
B_{q,\infty}^2\hookrightarrow B_{\infty,\infty}^{2-\f{3}{q}}.$$ Thus, according to
\eqref{5} and by applying H\"{o}lder's inequality, we
deduce that
\begin{equation*}
\begin{split}
\|\triangle f\|_{L^p(0,T; L^\infty)}&\le C\left(\int_0^T\|\nabla
f\|^{p\theta}_{B^{2-\f{3}{q}}_{\infty,\infty}}\|\nabla
f\|^{p(1-\theta)}_{B^{2-\f{3}{p}-\f{3}{q}}_{\infty,\infty}}dt\right)^{\f{1}{p}}\\
&\le
C\left(\int_0^T\|f\|^{p\theta}_{W^{3,q}}\|f\|^{p(1-\theta)}_{B_{q,p}^{3(1-\f{1}{p})}}dt\right)^{\f{1}{p}}\\
&\le CT^{\f{1}{3}-\f{1}{q}}\|f\|^{1-\theta}_{L^\infty(0,T;
B_{q,p}^{3(1-\f{1}{p})})}\|f\|^\theta_{L^p(0,T; W^{3,q})}.
\end{split}
\end{equation*}
The proof is complete.
\end{proof}

\bigskip

\section{Local Existence}

In this section, we prove the local existence and uniqueness of strong solution in Theorem \ref{T1}. The proof will be divided into several steps, including constructing the approximate solutions by iteration, obtaining the uniform estimate, showing the convergence, consistency and uniqueness.

\subsection{Construction of approximate solutions}
We initialize the construction of approximate solutions by setting
$$\d^0:= \d_0,\quad \u^0:=\u_0.$$
 For given $(\u^n, \d^n, P^n)$,  the
Stokes equations \eqref{e41} and the parabolic equation \eqref{e42}
enable us to define $(\u^{n+1}, \d^{n+1}, P^{n+1})$ as the global
solution of
\begin{equation}\label{e5}
\begin{cases}
\f{\partial\u^{n+1}}{\partial t}-\D\u^{n+1}+\nabla
P^{n+1}=-\u^n\cdot\nabla\u^n-\nabla\cdot\big((\nabla\d^n)^\top\nabla\d^n\big),\\
\f{\partial\d^{n+1}}{\partial
t}-\D\d^{n+1}=-\u^{n}\cdot\nabla \d^{n}+|\nabla\d^n|^2\d^n,\\
\Dv\u^{n+1}=0, \quad \int_{\O}P^{n+1}d\x=0
\end{cases}
\end{equation}
with the initial-boundary conditions:
$$(\u^{n+1},\d^{n+1})|_{t=0}=(\u_0,\d_0), \quad (\u^{n+1},\partial_{\bf \nu}\d^{n+1})|_{\partial\O}=(0,0).$$

According to Theorems \ref{T3}-\ref{T4}, an argument by
induction yields a sequence
$$\{(\u^n, \d^n, \\P^n)\}_{n\in {\mathbb N}} \subset M^{p,q}_T$$
 for all positive $T$.

\subsection{ Uniform estimate for some small fixed time $T$}   We aim
at finding a positive time $T$ independent of $n$ for which
$\{(\u^n, \d^n, P^n)\}_{n\in {\mathbb N}} $ is uniformly bounded in
the space $M^{p,q}_T$. Indeed, applying Theorem \ref{T4} to
\begin{equation*}
\begin{cases}
\f{\partial\u^{n+1}}{\partial t}-\D\u^{n+1}+\nabla
P^{n+1}=-\u^n\cdot\nabla\u^n-\nabla\cdot\big((\nabla\d^n)^\top\nabla\d^n\big),\\
\Dv\u^{n+1}=0, \quad \int_{\O}P^{n+1}d\x=0,\\
\u^{n+1}|_{t=0}=\u_0, \quad \u^{n+1}|_{\partial\O}=0
\end{cases}
\end{equation*}
 and Theorem \ref{T3} to
\begin{equation*}
\begin{cases}
\f{\partial\d^{n+1}}{\partial
t}-\D\d^{n+1}=-\u^{n}\cdot\nabla \d^{n}+|\nabla\d^n|^2\d^n,\\
\d^{n+1}|_{t=0}=\d_0, \quad \partial_{\bf \nu}\d^{n+1}|_{\partial\O}=0,
\end{cases}
\end{equation*}
 we obtain
\begin{equation}\label{3}
\begin{split}
&\|\u^{n+1}(T)\|_{D_{A_q}^{1-\f{1}{p},p}}+\left(\int_0^T\left\|\left(\nabla
P^{n+1}, \u^{n+1}, \D\u^{n+1},
\frac{\partial\u^{n+1}}{\partial t}\right)\right\|^p_{L^q}dt\right)^{\f{1}{p}}\\
&\le
C\left(\|\u_0\|_{D_{A_q}^{1-\f{1}{p},p}}+\left(\int_0^T\|\u^n\cdot\nabla\u^n+\nabla\cdot\big((\nabla\d^n)^\top\nabla\d^n\big)\|_{L^q}^pdt\right)^{\f{1}{p}}\right),
\end{split}
\end{equation}
and
\begin{equation}\label{4}
\begin{split}
&\|\d^{n+1}(T)\|_{B_{q,p}^{3(1-\f{1}{p})}}+\|\d^{n+1}\|_{\mathcal{W}(0,T)}\\
&\le
C\left(\|\d_0\|_{B_{q,p}^{3(1-\f{1}{p})}}+\|-\u^{n}\cdot\nabla \d^{n}+|\nabla\d^n|^2\d^n\|_{L^p(0,T;
L^q)}\right).
\end{split}
\end{equation}
 Now define
\begin{equation*}
\begin{split}
U^n(t):=&\|\u^n\|_{L^\infty(0,t; D_{A_q}^{1-\f{1}{p},p})}+\|\u^n\|_{L^p(0,t;
W^{2,q})}+\|\frac{\partial\u^n}{\partial t}\|_{L^p(0,t; L^q)}\\
&+\|\d^n\|_{L^\infty(0,t;
B_{q,p}^{3(1-\f{1}{p})})}+\|\d^n\|_{\mathcal{W}(0,t)},
\end{split}
\end{equation*}
 and
$$U^0:=\|\u_0\|_{D_{A_q}^{1-\f{1}{p},p}}+\|\d_0\|_{B_{q,p}^{3(1-\f{1}{p})}}.$$
Hence, from \eqref{3} and \eqref{4}, we get, using Lemmas
\ref{l1}-\ref{l3},
\begin{equation}\label{6}
\begin{split}
U^{n+1}(t)&\le C\Big(U^0+\|\u^n\|_{L^\infty(0,t;
L^q)}\|\nabla\u^n\|_{L^p(0,t; L^\infty)}\\
&\qquad\qquad+2\|\nabla\d^n\|_{L^\infty(0,t;
L^q)}\|\triangle\d^n\|_{L^p(0,t; L^\infty)}\\
&\qquad\qquad+\|\u^n\|_{L^\infty(0,t; L^q)}\|\nabla\d^n\|_{L^p(0,t;
L^\infty)}\\
&\qquad\qquad+\|\d^n\|_{L^\infty(0,t;
L^\infty)}\|\nabla\d^n\|_{L^\infty(0,t;
L^q)}\|\nabla\d^n\|_{L^p(0,t;
L^\infty)}\Big)\\
&\le
C\left(U^0+2t^{\f{1}{2}-\f{3}{2q}}\left(U^n(t)\right)^2+2t^{\f{1}{3}-\f{1}{q}}\left(U^n(t)\right)^2+t^{\f{1}{2}-\f{3}{2q}}\left(U^n(t)\right)^3\right).
\end{split}
\end{equation}
Moreover, if we assume that $U^n(t)\le 4CU^0$ on $[0,T]$ with
\begin{equation}\label{TTT1}
\begin{split}
&0<T\leq \left(\frac{3}{64C^2U^0+64C^3(U^0)^2}\right)^{\frac{3q}{q-3}}\leq1, \ \textrm{or}\\
& 1<T\leq\left(\frac{3}{64C^2U^0+64C^3(U^0)^2}\right)^{\frac{2q}{q-3}},
\end{split}
\end{equation}
then a direct computation yields
\begin{equation*}
U^{n+1}(t)\le 4CU^0 \quad\textrm{on}\ [0,T].
\end{equation*}

From \eqref{3}-\eqref{6}, we conclude
that the sequence $\{(\u^n, \F^n, P^n)\}_{n=1}^\infty$ is
uniformly bounded in $M^{p,q}_T$. More precisely, we have
\begin{Lemma}
For all $t\in [0,T]$ with $T$ satisfying \eqref{TTT1},
\begin{equation*}
U^n(t)\le 4CU^0.
\end{equation*}
\end{Lemma}

\subsection{Convergence of the approximation sequence}
\begin{Lemma}
 $\{(\u^n, \F^n, P^n)\}_{n=1}^\infty$ is a
Cauchy sequence in  $M_{T_0}^{p,q}$ and thus converges.
\end{Lemma}
\begin{proof}
Let
$$\bar{\u}^n:=\u^{n+1}-\u^n, \quad \bar{\F}^n:=\F^{n+1}-\F^n, \quad \bar{P}^n:=P^{n+1}-P^n.$$
Then, the triplet $(\bar{\u}^n, \bar{\F}^n, \bar{P}^n)$ satisfies
\begin{equation}\label{u}
\begin{cases}
\f{\partial\bar{\u}^{n}}{\partial t}-\D\bar{\u}^{n}+\nabla \bar{P}^{n}
=-\bar{\u}^{n-1}\cdot\nabla\u^{n}-\u^{n-1}\cdot\nabla\bar{\u}^{n-1}-\Dv\big((\nabla\bar{\F}^{n-1})^\top\nabla\F^{n}\big)
\\
\qquad\qquad\qquad\qquad\qquad-\Dv\big((\nabla\F^{n-1})^\top\nabla\bar{\F}^{n-1}\big), \\
\f{\partial\bar{\F}^{n}}{\partial
t}-\D\bar{\F}^{n}=-\bar{\u}^{n-1}\cdot\nabla \F^{n}-\u^{n-1}\cdot\nabla\bar{\F}^{n-1}+|\nabla\d^n|^2\bar{\d}^{n-1}\\
\qquad\qquad\qquad\quad
+\left((\nabla\d^n+\nabla\d^{n-1}):\nabla\bar{\d}^{n-1}\right)\d^{n-1},\\
\Dv\bar{\u}^{n}=0, \quad
\int_\O \bar{P}^n d\x=0
\end{cases}
\end{equation}
with the initial-boundary conditions:
$$(\bar{\u}^n,\bar{\d}^n)|_{t=0}=(0,0), \quad (\bar{\u}^n,\partial_{\bf\nu}\bar{\d}^n)|_{\partial\O}=(0,0).$$

 Define
\begin{equation*}
\begin{split}
 \bar{U}^n(t):=&\|\bar{\u}^n\|_{L^\infty(0,t;
D_{A_q}^{1-\f{1}{p},p})}+\|\bar{\u}^n\|_{L^p(0,t;
W^{2,q})}+\|\frac{\partial\bar{\u}^n}{\partial t}\|_{L^p(0,t; L^q)}\\
&+\|\nabla\bar{P}^n\|_{L^p(0,t; L^q)}+\|\bar{\F}^n\|_{L^\infty(0,t;
B_{q,p}^{3(1-\f{1}{p})})}+\|\bar{\F}^n\|_{\mathcal{W}(0,t)}.
\end{split}
\end{equation*}
By using Lemmas \ref{l1}-\ref{l3}, we obtain the following
estimates:
\begin{equation}\label{8}
\begin{split}
&\|\bar{\u}^{n-1}\cdot\nabla
\u^{n}+\u^{n-1}\cdot\nabla\bar{\u}^{n-1}\|_{L^p(0,t; L^q)}\\
&\le\|\bar{\u}^{n-1}\|_{L^\infty(0,t; L^q)}\|\nabla\u^n\|_{L^p(0,t;
L^\infty)}+\|\u^{n-1}\|_{L^\infty(0,t;
L^q)}\|\nabla\bar{\u}^{n-1}\|_{L^p(0,t; L^\infty)}\\
&\le
4CU^0\left(t^{\f{1}{2}-\f{3}{2q}}\|\bar{\u}^{n-1}\|_{L^\infty(0,t;
L^q)}+\|\nabla\bar{\u}^{n-1}\|_{L^p(0,t; L^\infty)}\right),
\end{split}
\end{equation}
\begin{equation}\label{9}
\begin{split}
&\|\Dv((\nabla\bar{\F}^{n-1})^\top\nabla\F^{n})
+\Dv((\nabla\F^{n-1})^\top\nabla\bar{\F}^{n-1})\|_{L^p(0,t;
L^q)}\\
&\le\|\nabla\F^n\|_{L^\infty(0,t;
L^q)}\|\triangle\bar{\F}^{n-1}\|_{L^p(0,t;
L^\infty)}+\|\nabla\bar{\F}^{n-1}\|_{L^\infty(0,t;
L^q)}\|\triangle\F^n\|_{L^p(0,t; L^\infty)}\\
&\quad+\|\nabla\bar{\F}^{n-1}\|_{L^\infty(0,t;
L^q)}\|\triangle\F^{n-1}\|_{L^p(0,t;
L^\infty)}+\|\nabla\F^{n-1}\|_{L^\infty(0,t;
L^q)}\|\triangle\bar{\F}^{n-1}\|_{L^p(0,t; L^\infty)}\\
&\le 8CU^0\left(\|\triangle\bar{\F}^{n-1}\|_{L^p(0,t;
L^\infty)}+t^{\f{1}{3}-\f{1}{q}}\|\nabla\bar{\F}^{n-1}\|_{L^\infty(0,t;
L^q)}\right),
\end{split}
\end{equation}
\begin{equation}\label{10}
\begin{split}
&\|\bar{\u}^{n-1}\cdot\nabla\F^n+\u^{n-1}\cdot\nabla\bar{\F}^{n-1}\|_{L^p(0,t;
L^q)}\\&\le\|\bar{\u}^{n-1}\|_{L^\infty(0,t;
L^q)}\|\nabla\F^n\|_{L^p(0,t; L^\infty)}+\|\u^{n-1}\|_{L^\infty(0,t;
L^q)}\|\nabla\bar{\F}^{n-1}\|_{L^p(0,t; L^\infty)}\\
&\le 4CU^0\left(t^{\f{1}{2}-\f{3}{2q}}\|\bar{\u}
^{n-1}\|_{L^\infty(0,t; L^q)}+\|\nabla\bar{\F}^{n-1}\|_{L^p(0,t;
L^\infty)}\right),
\end{split}
\end{equation}
and
\begin{equation}\label{11}
\begin{split}
&\||\nabla\d^n|^2\bar{\d}^{n-1}+\left((\nabla\d^n+\nabla\d^{n-1}):\nabla\bar{\d}^{n-1}\right)\d^{n-1}\|_{L^p(0,t;
L^q)}\\
&\le \left(4CU^0\right)^2\left(t^{\f{1}{2}-\f{3}{2q}}\|\bar{\d}^{n-1}\|_{L^\infty(0,t;L^\infty)}+2\|\nabla\bar{\F}^{n-1}\|_{L^p(0,t;
L^\infty)}\right).
\end{split}
\end{equation}
Applying Theorems \ref{T3}-\ref{T4} to \eqref{u},  with the help of
\eqref{8}-\eqref{11}, we have
\begin{equation}\label{12}
\begin{split}
 \bar{U}^n(t)&\le 4CU^0\Big(2t^{\f{1}{2}-\f{3}{2q}}\|\bar{\u}^{n-1}\|_{L^\infty(0,t;
L^q)}+\|\nabla\bar{\u}^{n-1}\|_{L^p(0,t;
L^\infty)}+2\|\triangle\bar{\F}^{n-1}\|_{L^p(0,t;
L^\infty)}\\
&\qquad\qquad+2t^{\f{1}{3}-\f{1}{q}}\|\nabla\bar{\F}^{n-1}\|_{L^\infty(0,t;
L^q)}+\|\nabla\bar{\F}^{n-1}\|_{L^p(0,t;
L^\infty)}\\
&\qquad\qquad+4CU^0\big(t^{\f{1}{2}-\f{3}{2q}}\|\bar{\F}^{n-1}\|_{L^\infty(0,t;
L^\infty)}+2\|\nabla\bar{\F}^{n-1}\|_{L^p(0,t; L^\infty)}\big)\Big).
\end{split}
\end{equation}
Combining \eqref{12} and Lemmas \ref{l1}-\ref{l3}, one has
\begin{equation*}
 \bar{U}^n(t)\le 16CU^0\left((1+3CU^0)t^{\f{1}{2}-\f{3}{2q}}+t^{\f{1}{3}-\f{1}{q}}\right)\bar{U}^{n-1}(t).
\end{equation*}
Thus, if we choose $T_0$ satisfying \eqref{TTT1} such that the condition
$$16CU^0(2+3CU^0){T_0}^{\f{1}{3}-\f{1}{q}}\le\f{1}{2},\ \ \textrm{or}\quad 16CU^0(2+3CU^0){T_0}^{\f{1}{2}-\f{3}{2q}}\le\f{1}{2}$$ is fulfilled, it is
clear that $\{(\u^n, \F^n, P^n)\}_{n=1}^\infty$ is a Cauchy
sequence in $M_{T_0}^{p,q}$ and thus converges in $M_{T_0}^{p,q}$.
\end{proof}

\subsection{The limit is a solution}
Since $\{(\u^n, \F^n, P^n)\}_{n=1}^\infty$ is a Cauchy sequence in $M_{T_0}^{p,q}$, then it converges.
 Let $(\u,\F, P)\in M_{{T_0}}^{p,q}$ be the limit of the sequence $\{(\u^n, \F^n, P^n)\}_{n=1}^\infty$ in $M_{T_0}^{p,q}$.
 We claim all those nonlinear terms in
\eqref{e5} converge to their corresponding terms in \eqref{e2} in
$\big(L^p(0,{T_0}; L^q(\O))\big)^3$.  Indeed, due to the convergence
of $\u^n$ to $\u$ in $ M_{{T_0}}^{p,q}$, we have
\begin{equation*}
\begin{split}
&\|\u^n\cdot\nabla\u^n-\u\cdot\nabla\u\|_{L^p(0,T_0;
L^q)}\\&=\|(\u^n-\u)\cdot\nabla\u^n+\u\cdot\nabla(\u^n-\u)\|_{L^P(0,{T_0};
L^q)}\\
&\le\|\u^n-\u\|_{L^\infty(0,{T_0}; L^q)}\|\nabla\u^n\|_{L^p(0,{T_0};
L^\infty)}+\|\u\|_{L^\infty(0,{T_0};
L^q)}\|\nabla\u^n-\nabla\u\|_{L^p(0,{T_0}; L^\infty)}\\
&\le 4CU^0{T_0}^{\f{1}{2}-\f{3}{2q}}\|\u^n-\u\|_{M_{{T_0}}^{p,q}}
+C{T_0}^{\f{1}{2}-\f{3}{2q}}\|\u\|_{L^\infty(0,{T_0};
L^q)}\|\u^n-\u\|_{M_{{T_0}}^{p,q}}\\
&\rightarrow 0.
\end{split}
\end{equation*}
  Hence,
$$\u^n\cdot\nabla\u^n\rightarrow\u\cdot\nabla\u\quad\textrm{in}\
\big(L^p(0,{T_0}; L^q(\O))\big)^3.$$ Similarly, we have
$$\nabla\cdot\big((\nabla\d^n)^\top\nabla\d^n\big)\rightarrow\nabla\cdot\big((\nabla\d)^\top\nabla\d\big)
\quad\textrm{in}\ \big(L^p(0,{T_0}; L^q(\O))\big)^3;$$
$$\u^{n}\cdot\nabla
\d^{n}\rightarrow\u\cdot\nabla\d\quad\textrm{in}\ \big(L^p(0,{T_0};
L^q(\O))\big)^3.$$ Since
\begin{equation*}
\begin{split}
&\||\nabla\d^n|^2\d^n- |\nabla\d|^2\d\|_{L^p(0,{T_0};L^q)}\\
&\leq \||\nabla\d^n|^2\d^n- |\nabla\d^n|^2\d\|_{L^p(0,{T_0};L^q)}+\||\nabla\d^n|^2\d-|\nabla\d|^2\d\|_{L^p(0,{T_0};L^q)}\\
&\leq \|\d^n-\d\|_{L^\infty(0,T_0;L^\infty)}\|\nabla\d^n\|_{L^\infty(0,T_0;L^q)}\|\nabla\d^n\|_{L^p(0,T_0;L^\infty)}\\
&\quad+\|\d\|_{L^\infty(0,T_0;L^\infty)}\|\nabla\d^n+\nabla\d\|_{L^\infty(0,T_0;L^q)}\|\nabla\d^n-\nabla\d\|_{L^p(0,T_0;L^\infty)}\\
&\leq \left((4CU^0)^2+\|\d\|_{L^\infty(0,T_0;L^\infty)}
(4CU^0+\|\nabla\d\|_{L^\infty(0,T_0;L^q)})\right)T_0^{\f{1}{2}-\f{3}{2q}}\|\d^n-\d\|_{M_{{T_0}}^{p,q}}\\
&\rightarrow 0,
\end{split}
\end{equation*}
then
$$|\nabla\d^n|^2\d^n\rightarrow |\nabla\d|^2\d\quad\textrm{in}\
\big(L^p(0,{T_0}; L^q(\O))\big)^3.$$ Thus, taking the limit as
$n\rightarrow\infty$ in \eqref{e5}, we conclude that \eqref{e2}
holds in $L^p(0,{T_0}; L^q(\O))$, and hence almost everywhere in
$\O\times (0,{T_0})$.

Multiply the $\d$-system, i.e., \eqref{e22} by $\d$, we obtain
\begin{equation*}
\frac{1}{2}\f{\partial|\d|^2}{\partial t}+\frac{1}{2}\u\cdot\nabla(|\d|^2)=\triangle\d\cdot\d+|\nabla\d|^2|\d|^2.
\end{equation*}
Since $$\triangle(|\d|^2)=2|\nabla\d|^2+2\d\cdot(\triangle\d),$$ then it follows that
\begin{equation*}
\frac{1}{2}\f{\partial|\d|^2}{\partial t}+\frac{1}{2}\u\cdot\nabla(|\d|^2)=\frac{1}{2}\triangle(|\d|^2)-|\nabla\d|^2+|\nabla\d|^2|\d|^2.
\end{equation*}
Therefore, it is easy to deduce that
\begin{equation}\label{s}
\f{\partial(|\d|^2-1)}{\partial t}-\triangle(|\d|^2-1)+\u\cdot\nabla(|\d|^2-1)-2|\nabla\d|^2(|\d|^2-1)=0.
\end{equation}
Multiplying \eqref{s} by $(|\d|^2-1)$ and then integrating over $\Omega$, using \eqref{e23} and \eqref{bc}, we get the following inequality:
\begin{equation}\label{ss}
\begin{split}
\frac{d}{dt}\int_{\Omega} (|\d|^2-1)^2d\x &
\le 4\int_{\Omega}|\nabla\d|^2(|\d|^2-1)^2d\x\\
 &\le 4\|\nabla\d\|_{L^\infty}^2\int_{\Omega}(|\d|^2-1)^2d\x.
\end{split}
\end{equation}
Remark that interpolating between $L^\infty(0,T_0; W^{1,q})$ and $L^p(0,T_0;W^{3,q})$  shows that for some  positive $\beta>\frac{1}{2}$,
 $\d$ belongs to $L^2(0,{T_0}; H^{2+\beta})$
and that $\|\nabla\d\|_{L^\infty}^2\in L^1(0,T_0)$. Notice that
$$\int_{\Omega}(|\d|^2-1)^2 d\x=0, \quad\text{at time}\quad t=0.$$
Thus, using \eqref{ss} together with Gr\"{o}nwall's inequality, it yields $|\d|=1$
 in $\O\times (0, T_0)$.

\subsection{Uniqueness}   Let $(\u_1, \F_1, P_1)$ and $(\u_2,
\F_2, P_2)$ be two solutions to \eqref{e2} with the initial-boundary
conditions \eqref{ic}-\eqref{bc}. Denote
$$\bar{\u}=\u_1-\u_2,\quad \bar{\F}=\F_1-\F_2,\quad \bar{P}=P_1-P_2.$$ Note
that the triplet $(\bar{\u}, \bar{\F}, \bar{P})$ satisfies the
following system:
\begin{equation*}
\begin{cases}
\frac{\partial\bar{\u}}{\partial t}-\D\bar{\u}+\nabla
\bar{P}=-\bar{\u}\cdot\nabla\u_1-\u_2\cdot\nabla\bar{\u}-\Dv\big((\nabla\d_1)^\top\nabla\bar{\d}\big)-\Dv\big((\nabla\bar{\d})^\top\nabla\d_2\big),\\
\frac{\partial\bar{\d}}{\partial t}-\D\bar{\d}=-\u_1\cdot\nabla\bar{\d}-\bar{\u}\cdot\nabla\d_2+|\nabla\d_1|^2\bar{\d}+\left((\nabla\d_1+\nabla\d_2):\nabla\bar{\d}\right)\d_2,\\
\Dv\bar{\u}=0, \quad \int_\O \bar{P}\ d\x=0
\end{cases}
\end{equation*}
with the initial-boundary conditions:
$$(\bar{\u},\bar{\F})|_{t=0}=(0,0), \quad (\bar{\u},\partial_{\bf \nu}\bar{\d})|_{\partial\O}=(0,0).$$

Define
\begin{equation*}
\begin{split}
X(t):=&\|\bar{\u}\|_{L^\infty(0,t;
D_{A_q}^{1-\f{1}{p},p})}+\|\bar{\u}\|_{L^p(0,t;
W^{2,q})}+\|\frac{\partial\bar{\u}}{\partial t}\|_{L^p(0,t; L^q)}\\
&+\|\nabla\bar{P}\|_{L^p(0,t;L^q)}+\|\bar{\F}\|_{L^\infty(0,t;
B_{q,p}^{3(1-\f{1}{p})})}+\|\bar{\F}\|_{\mathcal{W}(0,t)}.
\end{split}
\end{equation*}
Thus, repeating the arguments in \eqref{8}-\eqref{11}, we have
\begin{equation*}
\begin{split}
X(t)&\le 16CU^0\left((1+3CU^0)t^{\f{1}{2}-\f{3}{2q}}+t^{\f{1}{3}-\f{1}{q}}\right)X(t)\\
&\le\f{1}{2}X(t).
\end{split}
\end{equation*}
Hence, $X(t)=0$ for all $t\in [0,{T_0}]$, which guarantees the
uniqueness on the time interval $[0,{T_0}]$.

\bigskip
\section{Global Existence}

In this section, we prove that, if  the initial data is
sufficiently small, the local solution established in the previous
section is indeed global in time. To this end, we first denote
by $T^*$ the maximal time of existence for $(\u, \F, P)$. Define
the function $H(t)$ as
\begin{equation*}
\begin{split}
H(t):=&\|\u\|_{L^\infty(0,t;
D_{A_q}^{1-\f{1}{p},p})}+\|\u\|_{L^p(0,t;
W^{2,q})}+\|\frac{\partial\u}{\partial t}\|_{L^p(0,t;
L^q)}\\&\quad+\|\nabla P\|_{L^p(0,t; L^q)}+\|\F\|_{L^\infty(0,t;
B_{q,p}^{3(1-\f{1}{p})})}+\|\F\|_{\mathcal{W}(0,t)}.
\end{split}
\end{equation*}

To extend the local solution, we need to control the maximal time
$T^*$ only in terms of the initial data. For this purpose, it is
obvious to observe that $H(t)$ is an increasing and continuous
function in $[0, T^*)$, and for all $t\in [0,T^*)$, we have, using
Theorems \ref{T3}-\ref{T4},
\begin{equation}\label{18}
\begin{split}
H(t)\le &C\Big(U^0+\|\u\cdot\nabla\u\|_{L^p(0,t;
L^q)}+\|\nabla\cdot\big((\nabla\d)^\top\nabla\d\big)\|_{L^p(0,t;
L^q)}\\
&\qquad\quad+\|\u\cdot\nabla\d\|_{L^p(0,t; L^q)}+\||\nabla\d|^2\d\|_{L^p(0,t; L^q)}\Big).
\end{split}
\end{equation}
On the other hand, Lemmas \ref{l1}-\ref{l3} imply that
\begin{equation}\label{20}
\begin{split}
\|\u\cdot\nabla\u\|_{L^p(0,t; L^q)}&\le \|\u\|_{L^\infty(0,t;
L^q)}\|\nabla\u\|_{L^p(0,t; L^\infty)}\\&\le
Ct^{\f{1}{2}-\f{3}{2q}}H^2(t),
\end{split}
\end{equation}
\begin{equation}\label{21}
\begin{split}
\|\nabla\cdot\big((\nabla\d)^\top\nabla\d\big)\|_{L^p(0,t; L^q)}&\le
C\|\nabla\F\|_{L^\infty(0,t; L^q)}\|\triangle\F\|_{L^p(0,t;
L^\infty)}\\&\leq Ct^{\f{1}{3}-\f{1}{q}}H^2(t),
\end{split}
\end{equation}
\begin{equation}\label{22}
\begin{split}
\|\u\cdot\nabla\F\|_{L^p(0,t; L^q)} &\le \|\u\|_{L^\infty(0,t;
L^q)}\|\nabla\F\|_{L^p(0,t; L^\infty)}\\&\le
Ct^{\f{1}{2}-\f{3}{2q}}H^2(t),
\end{split}
\end{equation}
and for the fact that $|\d|=1$, we have
\begin{equation}\label{221}
\begin{split}
\||\nabla\d|^2\d\|_{L^p(0,t; L^q)} &\le \|\nabla\d\|_{L^\infty(0,t;
L^q)}\|\nabla\d\|_{L^p(0,t; L^\infty)}\\&\le
Ct^{\f{1}{2}-\f{3}{2q}}H^2(t).
\end{split}
\end{equation}

Substituting \eqref{20}-\eqref{221} into \eqref{18}, we get
\begin{equation}\label{23}
H(t)\le C\left(U^0+(3t^{\f{1}{2}-\f{3}{2q}}+t^{\f{1}{3}-\f{1}{q}})H^2(t)\right).
\end{equation}

Assume that $T$ is the smallest number such that
$$H(T)=4CU^0.$$ This is possible because $H(t)$ is an increasing and
continuous function in time. Then,
$$H(t)<H(T)=4CU^0,\quad\textrm{for all}\quad t\in[0,T),$$ and
from \eqref{23}, we deduce that
$$16C^2U^0\left(3T^{\f{1}{2}-\f{3}{2q}}+T^{\f{1}{3}-\f{1}{q}}\right)\geq 3.$$
This implies that the maximal time of existence $T^\ast$ will go to
infinity when the initial data  approaches zero. More precisely,  we
can show that,  if the initial data is sufficiently small, the
solution exists globally in time. To this end, we need some other
estimates for the terms on the right side of \eqref{18}. Indeed, by
$W^{1,q}\hookrightarrow L^\infty \ \textrm {as}\ q>3$, we have
\begin{equation*}
\begin{split}
\|\u\cdot\nabla\u\|_{L^p(0,t; L^q)}&\le \|\u\|_{L^\infty(0,t;
L^q)}\|\nabla\u\|_{L^p(0,t; L^\infty)}\\
&\le CH^2(t).
\end{split}
\end{equation*}
Similarly,
$$\|\nabla\cdot\big((\nabla\d)^\top\nabla\d\big)\|_{L^p(0,t; L^q)},\  \|\u\cdot\nabla\F\|_{L^p(0,t; L^q)},\ \||\nabla\d|^2\d\|_{L^p(0,t; L^q)}\le CH^2(t).$$
Thus, \eqref{18} turns out to be
\begin{equation}\label{40}
H(t)\le C\left(U^0+H^2(t)\right).
\end{equation}
 Now we take $U^0$ sufficiently small such that $$U^0\leq \delta_0:=\frac{1}{4C^2}.$$ Then, we compute directly from
\eqref{40} and the continuity of $H(t)$ that
$$H(t)\leq \frac{1-\sqrt{1-4C^2U^0}}{2C}\leq \frac{1}{2C}$$ for all $t\in [0,T^*)$, which implies
$\|(\u, \F, P)\|_{M^{p,q}_{T^*}}$ bounded.
Hence, according to the local existence in the previous section, we
can extend the solution on $[0,T^*)$ to some larger interval $[0,
T^*+\eta)$ with $\eta>0$. This is impossible since
 $T^*$ is already the maximal time of existence. Hence, when the initial data are sufficiently small,
 the strong solution is indeed global in time.

 The proof of Theorem \ref{T1} is complete.

\bigskip
\section{Weak-Strong Uniqueness}

 The purpose of this section is to show
$\textit{Weak-Strong Uniqueness}$ in Theorem \ref{T2}. To this end,
we first formally deduce and obtain an energy estimate for the
strong solution to \eqref{e2}-\eqref{bc}.
\begin{Lemma}\label{l4}
Let $p,q$ satisfy the same conditions as Theorem \ref{T1} and $(\u,
\F, P)\in M_T^{p,q}$ be the unique solution to \eqref{e2} on
$\O\times [0, T]$. Then for any $0<t\leq T$, we have
\begin{equation*}
\begin{split}
&\int_\O(|\u(t)|^2+|\nabla\d(t)|^2) d\x+2\int_0^t\!\!\int_\O
(|\nabla\u|^2+|\triangle\d+|\nabla\d|^2\d|^2) d\x ds\\&
=\int_\O(|\u_0|^2+|\nabla\d_0|^2)d\x.
\end{split}
\end{equation*}
\end{Lemma}

\begin{proof}
Note that
$$\u\in C([0,T]; D_{A_q}^{1-\f{1}{p},p})\cap
L^p(0,T; W^{2,q}\cap W_0^{1,q}),$$
$$\d\in C([0,T]; B_{q,p}^{3(1-\f{1}{p})})\cap L^p(0,T; W^{3,q}),$$   $$D_{A_q}^{1-\f{1}{p},p}\hookrightarrow B^{2(1-\f{1}{p})}_{q,p}\cap
X^q\ \textrm{(see Proposition 2.5 in \cite{D})}, $$  where
$$X^q=\{{\bf z}\in L^q(\O)^3\mid \nabla\cdot {\bf z}=0\ \
\textrm{in}\ \Omega \ \ \textrm{and} \ {\bf z}\cdot{\bf n}=0\ \
\textrm{on}\
\partial\Omega\},$$
for some  positive $\beta>\frac{1}{2}$, it follows from the standard
interpolation
 inequalities that
 $$\u\in C([0,T]; H^\beta)\cap L^2(0,T; H^{1+\beta}), \ \u\in L^4(\O\times [0,T]),$$  $$\F\in C([0,T]; H^{1+\beta})\cap L^2(0,T; H^{2+\beta}),\ \nabla\d\in L^4(\O\times [0,T]).$$  This enables us to justify the following
computations.

Multiplying \eqref{e21} by $\u$, integrating over $\Omega$, we get
\begin{equation}\label{16}
\frac{1}{2}\f{d}{dt}\int_\O|\u|^2d\x+\int_\O|\nabla\u|^2d\x=-\int_\O\u\cdot\big((\nabla\d)^\top\triangle\d\big)\
d\x.
\end{equation}
 Here we have used the facts
$$\nabla\cdot(\nabla\d\odot\nabla\d)=\nabla\left(\frac{|\nabla\d|^2}{2}\right)+(\nabla\d)^\top\triangle\d,$$
and  $\nabla\cdot\u=0$ in $\Omega$, $\u=0$ on $\partial\Omega$,
as well as
$$\int_\O\u\cdot\nabla\u\cdot\u\ d\x=\int_\O\nabla P\cdot\u \ d\x=\int_\O \nabla\left(\frac{|\nabla\d|^2}{2}\right)\cdot\u\ d\x=0.$$
Multiplying \eqref{e22} by $-(\triangle\d+|\nabla\d|^2\d)$ and integrating over
$\Omega$, we obtain
\begin{equation*}
-\int_\O\frac{\partial\d}{\partial t}\cdot\triangle\d\ d\x-\int_\O
(\u\cdot\nabla\d)\cdot \triangle\d \ d\x=-\int_\O |\triangle\d+|\nabla\d|^2\d|^2
d\x.
\end{equation*}
Here we have used the fact that $|\d|=1$ to get
$$\left(\frac{\partial\d}{\partial t}+\u\cdot\nabla\d\right)\cdot|\nabla\d|^2\d
=\frac{1}{2}\left(|\nabla\d|^2\frac{\partial|\d|^2}{\partial t}+\u\cdot\nabla|\d|^2|\nabla\d|^2\right)=0.$$
Since $\partial_{\bf\nu} \d=0$ on $\partial\O$, integrating by parts, we have $$\int_\O\frac{\partial \d}{\partial t}\cdot\triangle\d \ d\x=-\frac{1}{2}\frac{d}{dt}\int_\O|\nabla\d|^2\ d\x.$$
Hence we obtain
\begin{equation}\label{171}
\frac{1}{2}\f{d}{dt}\int_\O|\nabla\d|^2 d\x+\int_\O
|\triangle\d+|\nabla\d|^2\d|^2 \ d\x=\int_\O
(\u\cdot\nabla\d)\cdot \triangle\d \ d\x.
\end{equation}
By adding \eqref{16} and \eqref{171}, we eventually get the identity:
\begin{equation}\label{167}
\frac{1}{2}\f{d}{dt}\int_\O(|\u|^2+|\nabla\F|^2) d\x+\int_\O
(|\nabla\u|^2+|\triangle\d+|\nabla\d|^2\d|^2) d\x=0,
\end{equation}
for all $t\in (0, T]$.

Integrating \eqref{167} over the time interval $[0,t]$, we obtain
the energy equality of this lemma.
\end{proof}
We remark that \eqref{167} is usually called the basic energy law
governing the system \eqref{e2}-\eqref{bc}. It reflects the energy
dissipation property of the flow of liquid crystals.

\vspace{3mm}


 Now we proceed to prove  $\textit{Weak-Strong
Uniqueness}$. In view of the regularity of $\u$, we deduce from the
weak formulation of \eqref{e21} that
\begin{equation}\label{27}
\begin{split}
&\int_\O \tilde{\u}\cdot \u \ d\x+\int_0^t\!\!\int_\O\nabla\tilde{\u}:\nabla\u \ d\x ds\\
&=\int_\O|\u_0|^2\
d\x-\int_0^t\!\!\int_\O(\nabla\tilde{\d})^\top\triangle\tilde{\d}\cdot
\u\ d\x ds +\int_0^t\!\!\int_\O \tilde{\u}\cdot\left(\f{\partial \u}{\partial
s}+\tilde{\u}\cdot\nabla \u\right)\ d\x ds.
\end{split}
\end{equation}

On the other hand, since $\u$ satisfies
\eqref{e21}, i.e.,
$$\f{\partial\u}{\partial t}=\D\u-\u\cdot\nabla\u-\nabla P-\nabla\left(\frac{|\nabla\d|^2}{2}\right)-(\nabla\d)^\top\triangle\d,$$ then we
have, from \eqref{27},
\begin{equation}\label{28}
\begin{split}
&\int_\O \tilde{\u}\cdot \u \ d\x-\int_\O|\u_0|^2\ d\x\\
&=-2\int_0^t\!\!\int_\O\nabla\tilde{\u}:\nabla\u\ d\x
ds-\int_0^t\!\!\int_\O(\nabla\tilde{\d})^\top\triangle\tilde{\d}\cdot\u\ d\x ds\\
&\quad-\int_0^t\!\!\int_\O\tilde{\u}\cdot\left(\u\cdot\nabla\u+(\nabla\d)^\top\triangle\d-\tilde{\u}\cdot\nabla\u\right)\
d\x ds.
\end{split}
\end{equation}
Similarly, in view of the regularity of $\d$, we have
\begin{equation}\label{271}
\begin{split}
&\int_\O\nabla\tilde{\d}:\nabla\d\ d\x-\int_\O|\nabla\d_0|^2\ d\x\\
&=\int_0^t\!\!\int_\O\big(-\tilde{\d}\cdot \triangle\d_s+\tilde{\u}\cdot\nabla\tilde{\d}\cdot
\triangle\d-\triangle\tilde{\d}\cdot\triangle\d-|\nabla\tilde{\d}|^2\tilde{\d}\cdot
\triangle\d\big)\ d\x ds.
\end{split}
\end{equation}
Taking advantage of \eqref{e22}, we obtain, from \eqref{271},
\begin{equation}\label{272}
\begin{split}
&\int_\O\nabla\tilde{\d}:\nabla\d\ d\x-\int_\O|\nabla\d_0|^2\ d\x\\
&=\int_0^t\!\!\int_\O\Big(-2\triangle\tilde{\d}\cdot
\triangle\d+\u\cdot\nabla\d\cdot\triangle\tilde{\d}+\tilde{\u}\cdot\nabla\tilde{\d}\cdot\triangle\d\\
&\qquad\qquad\qquad -|\nabla\d|^2\d\cdot
\triangle\tilde{\d}-|\nabla\tilde{\d}|^2\tilde{\d}\cdot\triangle\d\Big)\ d\x
ds.
\end{split}
\end{equation}
From \eqref{26}, \eqref{28}, \eqref{272} and the fact that $(\u, \d, P)$ \big(resp. $(\tilde{\u}, \tilde{\d}, \Pi )$\big) is a strong
solution (resp. weak solution) to \eqref{e2} with the
initial-boundary conditions
\eqref{ic}-\eqref{bc},
 we have
the following energy estimate of $(\u-\tilde{\u}, \d-\tilde{\d})$:
\begin{equation}\label{m}
\begin{split}
&\frac{1}{2}\int_\Omega(|\u-\tilde{\u}|^2+|\nabla\d-\nabla\tilde{\d}|^2)\ d\x\\&\leq
-\int_0^t\!\!\int_\O(|\nabla\u-\nabla\tilde{\u}|^2+|\triangle\d-\triangle\tilde{\d}|^2)\
d\x ds\\
&\quad-\int_0^t\!\!\int_\O\Big((\u-\tilde{\u})\cdot\nabla\u\cdot(\u-\tilde{\u})+(\nabla\d-\nabla\tilde{\d})\cdot\triangle\d\cdot(\u-\tilde{\u})\\
&\qquad\qquad\qquad-\u\cdot(\nabla\d-\nabla\tilde{\d})\cdot(\triangle\d-\triangle\tilde{\d})\\
&\qquad\qquad\qquad +(|\nabla\d|^2\d-|\nabla\tilde{\d}|^2\tilde{\d})\cdot(\triangle\d-\triangle\tilde{\d})\Big)\
d\x ds\\
&=-\int_0^t\!\!\int_\O(|\nabla\u-\nabla\tilde{\u}|^2+|\triangle\d-\triangle\tilde{\d}|^2)\
d\x ds+I,
\end{split}
\end{equation}
where
\begin{equation*}
\begin{split}
I=&-\int_0^t\!\!\int_\O\Big((\u-\tilde{\u})\cdot\nabla\u\cdot(\u-\tilde{\u})+(\nabla\d-\nabla\tilde{\d})\cdot\triangle\d\cdot(\u-\tilde{\u})\\
&\qquad\qquad\quad-\u\cdot(\nabla\d-\nabla\tilde{\d})\cdot(\triangle\d-\triangle\tilde{\d})\\
&\qquad\qquad\quad+(|\nabla\d|^2\d-|\nabla\tilde{\d}|^2\tilde{\d})\cdot(\triangle\d-\triangle\tilde{\d})\Big)\
d\x ds.
\end{split}
\end{equation*}
Next, we will estimate $I$ term by term. By the zero boundary
condition, we have
\begin{equation}\label{30}
\begin{split}
-\int_\O(\u-\tilde{\u})\cdot\nabla\u\cdot(\u-\tilde{\u})\ d\x
&=\int_\O(\u-\tilde{\u})\cdot\nabla\tilde{\u}\cdot \u\ d\x\\
&=\int_\O(\u-\tilde{\u})\cdot(\nabla\u-\nabla\tilde{\u})\cdot \u\
d\x\\&\leq \|\u\|_{L^\infty}\|\nabla\u-\nabla\tilde{\u}\|_{L^2}\|\u-\tilde{\u}\|_{L^2}\\
&\leq \frac{1}{2}\|\nabla\u-\nabla\tilde{\u}\|_{L^2}^2+\frac{\|\u\|_{L^\infty}^2}{2}\|\u-\tilde{\u}\|_{L^2}^2,
\end{split}
\end{equation}
\begin{equation}\label{31}
\begin{split}
-\int_\O(\nabla\d-\nabla\tilde{\d})\cdot\triangle\d\cdot(\u-\tilde{\u})\ d\x
&\leq \|\triangle\d\|_{L^\infty}\|\u-\tilde{\u}\|_{L^2}\|\nabla\d-\nabla\tilde{\d}\|_{L^2}\\
&\leq \frac{\|\triangle\d\|_{L^\infty}}{2}(\|\u-\tilde{\u}\|_{L^2}^2+\|\nabla\d-\nabla\tilde{\d}\|_{L^2}^2),
\end{split}
\end{equation}
\begin{equation}\label{32}
\begin{split}
\int_\O\u\cdot(\nabla\d-\nabla\tilde{\d})\cdot(\triangle\d-\triangle\tilde{\d})\
d\x&\leq \|\u\|_{L^\infty}\|\triangle\d-\triangle\tilde{\d}\|_{L^2}\|\nabla\d-\nabla\tilde{\d}\|_{L^2}\\
&\leq \frac{1}{2}\|
\triangle\d-\triangle\tilde{\d}\|_{L^2}^2+\frac{\|\u\|_{L^\infty}^2}{2}\|\nabla\d-\nabla\tilde{\d}\|_{L^2}^2,
\end{split}
\end{equation}
\begin{equation}\label{34}
\begin{split}
&-\int_\O(|\nabla\d|^2\d-|\nabla\tilde{\d}|^2\tilde{\d})\cdot(\triangle\d-\triangle\tilde{\d})\
d\x\\&\leq
\|\nabla\d\|^2_{L^\infty}\|\d-\tilde{\d}\|_{L^2}\|\triangle\d-\triangle\tilde{\d}\|_{L^2}\\
&\qquad +\|\nabla\d
+\nabla\tilde{\d}\|_{L^\infty}\|\nabla\d-\nabla\tilde{\d}\|_{L^2}\|\triangle\d-\triangle\tilde{\d}\|_{L^2}\\
&\leq \frac{1}{2}
\|\triangle\d-\triangle\tilde{\d}\|_{L^2}^2+\|\nabla\d\|^4_{L^\infty}
\|\d-\tilde{\d}\|_{L^2}^2+\|\nabla\d+\nabla\tilde{\d}\|^2_{L^\infty}\|\nabla\d-\nabla\tilde{\d}\|^2_{L^2}.
\end{split}
\end{equation}
Then, we eventually get from \eqref{30}-\eqref{34} that
\begin{equation}\label{I}
\begin{split}
I&\leq \int_0^t
\left(\frac{1}{2}\|\nabla\u-\nabla\tilde{\u}\|_{L^2}^2+\|\triangle\d-\triangle\tilde{\d}\|_{L^2}^2+\frac{\|\u\|^2_{L^\infty}
+\|\triangle\d\|_{L^\infty}}{2}\|\u-\tilde{\u}\|_{L^2}^2\right.\\
&\qquad\qquad\left.+\big(\|\nabla\d+\nabla\tilde{\d}\|^2_{L^\infty}+\frac{\|\u\|^2_{L^\infty}
+\|\triangle\d\|_{L^\infty}}{2}\big)\|\nabla\d-\nabla\tilde{\d}\|_{L^2}^2\right.\\
&\qquad\qquad\left.+\|\nabla\d\|_{L^\infty}^4\|\d-\tilde{\d}\|^2_{L^2}\right)ds.
\end{split}
\end{equation}
Now, we wish to estimate $\|\d-\tilde{\d}\|_{L^2}$ . We write
\begin{equation}\label{w1911}
\partial_t(\d-\tilde{\d})+\u\cdot\nabla(\d-\tilde{\d})+(\u-\tilde{\u})\cdot\nabla\tilde{\d}=\triangle\d-\triangle\tilde{\d}+|\nabla\d|^2\d-|\nabla\tilde{\d}|^2\tilde{\d}.
\end{equation}
Multiply \eqref{w1911} by $\d-\tilde{\d}$ and integrate over
$\O\times (0,t)$, we have
\begin{equation*}
\begin{split}
&\frac{1}{2}\int_\O|\d-\tilde{\d}|^2\ d\x\\
&=-\int_0^t\int_{\O}(\u-\tilde{\u})\cdot\nabla\tilde{\d}\cdot(\d-\tilde{\d})\
d\x ds-\int_0^t\int_{\O}|\nabla\d-\nabla\tilde{\d}|^2\ d\x
ds\\
&\quad+\int_0^t\int_{\O}|\nabla\d|^2|\d-\tilde{\d}|^2\ d\x ds
+\int_0^t\int_{\O}(\nabla\d+\nabla\tilde{\d}):(\nabla\d-\nabla\tilde{\d})\
\tilde{\d}\cdot(\d-\tilde{\d})\ d\x ds.
\end{split}
\end{equation*}
 Using Sobolev's inequality $\|\u-\tilde{\u}\|_{L^6}\leq
C\|\nabla\u-\nabla\tilde{\u}\|_{L^2}$ and for some $\varepsilon>0$\
small enough, it is easy to get
\begin{equation}\label{w1922}
\begin{split}
&\frac{1}{2}\int_\O|\d-\tilde{\d}|^2\ d\x\\
&\leq
\int_0^t\left(C_\varepsilon\|\nabla\tilde{\d}\|_{L^3}^2+\|\nabla\d\|_{L^\infty}^2+\frac{1}{2}\right)\int_{\O}|\d-\tilde{\d}|^2\
d\x ds+\varepsilon\int_0^t\int_{\O}|\nabla\u-\nabla\tilde{\u}|^2\
\ d\x ds\\
&\quad+\int_0^t\left(1+\frac{\|\nabla\d+\nabla\tilde{\d}\|^2_{L^\infty}}{2}\right)\int_{\O}|\nabla\d-\nabla\tilde{\d}|^2\
d\x ds.
\end{split}
\end{equation}
Now we have from \eqref{m} \eqref{I} and \eqref{w1922} that
\begin{equation}\label{f}
\begin{split}
&\frac{1}{2}\int_\Omega(|\u-\tilde{\u}|^2+|\d-\tilde{\d})|^2+|\nabla\d-\nabla\tilde{\d})|^2)\
d\x\\& \leq
\int_0^t\frac{\|\u\|^2_{L^\infty}+\|\triangle\d\|_{L^\infty}}{2}\int_{\O}|\u-\tilde{\u}|^2\
d\x ds\\
&\quad
+\int_0^t\left(\frac{1}{2}+\|\nabla\d\|^4_{L^\infty}+C\|\nabla\tilde{\d}\|_{L^3}^2+\|\nabla\d\|^2_{L^\infty}\right)\int_{\O}|\d-\tilde{\d}|^2\
d\x ds\\
&\quad+
\int_0^t\left(1+\frac{\|\u\|_{L^\infty}^2+\|\triangle\d\|_{L^\infty}}{2}+\frac{3\|\nabla\d+\nabla\tilde{\d}\|^2_{L^\infty}}{2}\right)\int_{\O}|\nabla\d-\nabla\tilde{\d}|^2\
d\x ds
\\& \leq
C\int_0^t\va(s)\int_\Omega(|\u-\tilde{\u}|^2+|\d-\tilde{\d}|^2+|\nabla\d-\nabla\tilde{\d}|^2)\
d\x ds,
\end{split}
\end{equation}
where
$$\va(s)=1+\|\nabla\d\|^4_{L^\infty}+\|\nabla\d\|^2_{L^\infty}+\|\nabla\tilde{\d}\|^2_{L^\infty}+\|\u\|^2_{L^\infty}
+\|\triangle\d\|_{L^\infty}.$$

 Notice that
$\|\nabla\d\|^2_{L^\infty}, \|\u\|^2_{L^\infty},
\|\triangle\d\|_{L^\infty}\in L^1(0,t)$. Moreover, by applying the
quasi-linear equations of parabolic type estimates (cf. \cite{OAL}
Chapter VI, Section 2) to \eqref{e22},
 we see $\d(\cdot, t),\ \tilde{\d}(\cdot, t)\in C^{1, \alpha}$ with respect to the space
variables, for some $\alpha>0$, and its $C^{1, \alpha}$ norm is
independent of $t$. Then we have $\va(s)\in L^1(0,t)$. Applying
Gr\"{o}nwall's inequality to \eqref{f}, we obtain
$$\int_\Omega(|\u-\tilde{\u}|^2+|\d-\tilde{\d}|^2+|\nabla\d-\nabla\tilde{\d}|^2)\ d\x=0$$ for
all $t$. Thus,
 $\u=\tilde{\u}, \d=\tilde{\d} \ a.e. $ and $P=\Pi$ up to a constant in
$\O\times(0,T)$.

The proof of Theorem \ref{T2} is now complete.

\bigskip\bigskip

\section*{Acknowledgments}
X.  Li's research was supported in part by a joint project from the
NSAF of China (China Scholarship Council), and by Doctoral
Innovation Fund of Tsinghua University. D. Wang's research was
supported in part by the National Science Foundation under Grant
DMS-0906160 and by the Office of Naval Research under Grant
N00014-07-1-0668.

\end{document}